\documentclass[a4paper,12pt]{article}
\usepackage{amsrefs,amstext,amsmath,amsthm,amsfonts,amssymb,enumerate,amscd,makeidx,mathrsfs}
\usepackage[refpage,intoc]{nomencl}
\usepackage{hyperref}
\usepackage[dvipdf]{graphicx}
\usepackage[margin=3 true cm]{geometry}
\usepackage{todonotes}
\usepackage{xspace}

\newtheorem{theorem}{Theorem}[section]
\newtheorem{proposition}[theorem]{Proposition}
\newtheorem{lemma}[theorem]{Lemma}
\newtheorem{corollary}[theorem]{Corollary}

\newtheorem{definition}[theorem]{Definition}
\newtheorem{remark}[theorem]{Remark}

\makenomenclature

\title{Compactness of Schur A-multipliers and Haagerup Tensor Products}
\author{Weijiao He}

\begin{document}

\maketitle

\begin{abstract}
In this paper we study the connection between Haagerup tensor product and compactness of Schur $A$-multiplier. In particular, we give a new characterization of elementary $C^{\ast}$-algebra in terms of completely compact Schur $A$-multiplier.
\end{abstract}

\section{Introduction}

Schur multipliers, a class of maps generalising the operators of entry-wise (Schur) multiplication on finite matrices, were first abstractly studied by Grothendieck in \cite{M025}. Since then they have played an important role in operator theory. In the simplest situation they arise in the following manner: to a (discrete) set $X$ and a function $\phi: X \times X \to \mathbb{C}$, one associates an operator $S_{\phi}$ on the space of compact operators on the Hilbert space $\ell^2(X)$; if the resulting map is (completely) bounded, we call $S_{\phi}$ a Schur multiplier with symbol $\phi$. 

In \cite{MR1766604}, Hladnik studied an important class of Schur multipliers: compact Schur multipliers, i.e the map $S_{\phi}$ is a compact operator. Hladnik identified the space of Schur multipliers with Haagerup tensor product $c_0 \otimes_h c_0$.

Recently, in \cite{MTT16}, McKee, Todorov and Turowska generalised the notion of Schur multipliers to new setting, on which we will inverstigate the generalization of Hladnik's thorem. This paper is organised as following.

In section 2, we give some basic definitions which we used in this paper, including the definition of Schur $A$-multiplier.

In section 3, we prove that if either $X$ or $Y$ is not discrete measure space, then there is no non-zero compact Schur $A$-multiplier. By this result, we could restrict our attention to the case $X=Y=\mathbb{N}$.

In section 4, we study some properties of Haagerup tensor product. We prove a Theorem which based on the work of Smith \cite{MR1138841},  Ylinen \cite{MR0296716} and Saar (see \cite{compactness}), to get a complete relationship between Haagerup tensor product and completely compact maps on $C^{\ast}$-algebras, which will be very useful for the study of the compactness of Schur $A$-multipliers in the later section. Furthemore, that theorem gives a new characterization of the $C^{\ast}$-algebra of compact operators of some Hilbert space.

In second 5 and section 6, we study the complete compactness of Schur $A$-multipliers when $A$ is $\ast$-isomorphic to a subalgebra of $\mathcal{K}(H)$, and we prove a generalisation of Hladnik's Theorem.

In Section 7, it contains our main results. We study the relationship between complete compactness and compactness, in the end we use these relations and the results of previous sections to prove that the generalization of Hladnik's Theorem is true if and only if the $C^{\ast}$-algebra $A$ if $\ast$-isomorphic $\mathcal{B}(H)$ for some finite dimensional Hilbert space $H$.

\section{Revision of the Operator-Valued Schur Multiplier}

For any measure space ($Z,\lambda$) and Banach space $B$, we let $\mathfrak{L}_2(Z, B)$ denote the space of all square integrable $\lambda$-measurable functions from $Z$ into $B$. If $B=K$ for some Hilbert space $K$, then $\mathfrak{L}_2(Z, K)$ is a Hilbert space. Furthermore, for any Hilbert space $K$, we denote the space of bounded operators on $K$ by  $\mathcal{O}(K)$, and denote the space of compact operators on $K$ by  $\mathcal{O}_c(K)$ .

Let $(X,\mu)$ and $(Y,\nu)$ be standard measure spaces (see \cite{MTT16}), $H$ a separable Hilbert space, and $A \subseteq \mathcal{O}(H)$ a $C^{\ast}$-algebra. We write $\mathcal{O}_c=\mathcal{O}_c(\mathfrak{L}_2(X), \mathfrak{L}_2(Y))$ for the space of all compact operators from $\mathfrak{L}_2(X,H)$ into $\mathfrak{L}_2(Y,H)$. If $k \in \mathfrak{L}_2(Y \times X, \mathcal{O}(H))$ and $\xi \in \mathfrak{L}_2(X,H)$ then for almost all $y \in Y$, the function $x \to k(y,x) \xi(x)$ is weakly measurable; moreover
\begin{equation*}
 \int_X \|k(y,x)\xi(x)\| d \mu(x) \leq \|\xi\|_2  \text{\Huge{(}}\int_X \|k(y,x)\|^2 d \mu(x)\text{\Huge{)}}^{\frac{1}{2}}. 
\end{equation*}
Such functions $k$ will often be referred to as $kernels$. It follows that the formula
\begin{equation*}
(T_k\xi)(y)=\int_X k(y,x) \xi(x) d\mu(x) \ \ \ \ (y \in Y),
\end{equation*}
defines a (weakly measurable) function $T_k\xi: Y \to H$, and a bounded operator $T_k: \mathfrak{L}_2(Y, H) \to \mathfrak{L}_2(X, H)$. Moreover,  by  \cite{MTT16} we have  $\| T_k \| \leq \| k  \|_2$    and $ T_k$=0 if and only if $k=0$ almost everywhere.

 If $\mathcal{X}$ and $\mathcal{Y}$ are operator space, we denote the space of all completely bounded linear maps from $\mathcal{X}$ into $\mathcal{Y}$ by $CB(\mathcal{X},\mathcal{Y})$ and write $CB(\mathcal{X})=CB(\mathcal{X},\mathcal{X})$. For the background of operator spaces and completely bounded maps, we refer the reader to Section 1.2. In this thesis, if $f$ is a linear map from an operator space $\mathcal{X}$ into an operator space $\mathcal{Y}$, we use $f_n$ to denote the corresponding map from $M_n(\mathcal{X})$ into $M_n(\mathcal{Y})$.

Now we define
\begin{equation*}
\mathcal{S}_2(X \times Y, A)=\{T_k: k \in \mathfrak{L}_2(Y \times X, A)\}
\end{equation*}
and note that $\mathcal{S}_2(Y \times X, A)$ is a dense subspace of the minimal tensor product $\mathcal{O}_c \otimes A$, thus in particular it is an operator space. A function $\varphi: X \times Y \to CB(A, \mathcal{O}(H))$ will be called $pointwise \ measurable$ if, for every $a \in A$, the function $(x,y) \to \varphi(x,y)(a)$ from $X \times Y$ into $\mathcal{O}(H)$ is weakly measurable (\cite{MTT 16}). Let $\varphi: X \times Y \to CB(A, \mathcal{O}(H))$ be a bounded pointwise measurable function. For $k \in \mathfrak{L}_2(Y \times X, A)$, let $\varphi \cdot k: Y \times X \to \mathcal{O}(H)$ be the function given by
\begin{equation*}
(\varphi \cdot k)(y,x)=\varphi(x,y)(k(y,x)) \ \ \ \ ((y,x) \in Y \times X).
\end{equation*}
It is easy to show that $\varphi \cdot k$ is weakly measurable and $\| \varphi \cdot k \|_2 \leq \| \varphi \|_{\infty} \| k\|_2$ (\cite[Section 2]{MTT 16}). Let
\begin{equation*}
S_{\varphi}: \mathcal{S}_2(Y \times X, A) \to \mathcal{S}_2(Y \times X, \mathcal{O}(H))
\end{equation*}
be the linear map given by
\begin{equation*}
S_{\varphi}(T_k)=T_{\varphi \cdot k} \ \ \ ( k \in \mathfrak{L}_2(Y \times X, A)).
\end{equation*}

\begin{definition}
A bounded poinwise measurable map
\begin{equation*}
\varphi: X \times Y \to CB(A, \mathcal{O}(H))
\end{equation*}
will be called a Schur $A$-multiplier if the map $S_{\varphi}$ is completely bounded.
\end{definition}

Equivalently, a bounded pointwise measurable function $\varphi: X \times Y \to CB(A,\mathcal{O}(H))$ is a Schur $A$-multiplier if and only if the map $S_{\varphi}$ possesses a completely bounded extension to a map from $\mathcal{O}_c \otimes A$ into $\mathcal{O}_c \otimes \mathcal{O}(H)$ (which we will still denote by $S_{\varphi}$). 

\bigskip

 For the sake of convenience, we will not distinguish Schur $A$-multiplier $\varphi: X \times Y \to CB(A, \mathcal{O}(H))$ and the corresponding linear map
\begin{equation*}
S_{\varphi}: \mathcal{O}_c(\mathfrak{L}_2(X), \mathfrak{L}_2(Y)) \otimes A \to \mathcal{O}_c(\mathfrak{L}_2(X), \mathfrak{L}_2(Y)) \otimes \mathcal{O}(H),
\end{equation*}
when we use the terminology `Schur $A$-multiplier'.

Another important notion is complete compactness which is defined as follows (see \cite{completelycompact})

\begin{definition}
If $\mathcal{X}$ and $\mathcal{Y}$ are operator spaces, a completely bounded map $\Phi: \mathcal{X} \to \mathcal{Y}$ is called completely compact if for each $\epsilon >0$ there exists a finite dimensional subspace $F \subset \mathcal{Y}$ such that
\begin{equation*}
{\rm{dist}} (\Phi^{(m)}(x), M_m(F)) < \epsilon,
\end{equation*}
for every $x \in M_m(\mathcal{X})$ with $\|x\| \leq 1$ for every $m \in \mathbb{N}$.
\end{definition}
Let us recall that a completely bounded linear map which is approximated by a net of linear maps with finite rank in the complete bounded norm is completely compact \cite[Proposition 3.2]{compactness}. We will use this fact without reference frequently.

If $\mathcal{X}$ and $\mathcal{Y}$ are operator spaces, we  denote the set of compact (resp. completely compact) operators from $\mathcal{X}$ into $\mathcal{Y}$ by $\mathfrak{CO}(\mathcal{X,Y})$ (resp. $\mathfrak{CCO}(\mathcal{X,Y})$).

\begin{remark}\label{hfsadfjklhvuivrioreiu}
Let  $\mathcal{X}$, $\mathcal{Y}$ be operator spaces, $\varphi: \mathcal{X} \to \mathcal{Y}$ be completely bounded linear map.  If $\mathcal{Z} \subset \mathcal{Y}$ is operator space such that $\varphi(\mathcal{X}) \subset \mathcal{Z}$ and there is completely bounded map $h: \mathcal{Y} \to \mathcal{Z}$ with $h(d)=d$ for all $d \in \mathcal{Z}$, we define $\psi: \mathcal{X} \to \mathcal{Z}$ by $\psi(a)=\varphi(a)$ for all $a \in \mathcal{X}$, then since the composition of completely compact map and completely bounded map is completely compact, we conclude that $\varphi$ is completely compact if and only if $\psi$ is completely compact $($the proof is easy consequence of \cite[Proposition 3.2]{compactness}$)$. We will use this fact without reference in this paper.
\end{remark}

\begin{lemma}\label{negativecriteria}
Let $\mathcal{X}$ and $\mathcal{Y}$ be operator spaces, $\varphi: \mathcal{X} \to \mathcal{Y}$ a completely bounded linear map which is not completely compact. If $\mathcal{Z}$ is an operator space containing $\mathcal{X}$ as a subspace, and $f: \mathcal{Z} \to \mathcal{X}$ is a completely bounded surjective linear map such that $\|f\|_{cb}=1$ and $f|\mathcal{X}=I_{\mathcal{X}}$ (here $I_{\mathcal{X}}: \mathcal{X} \to \mathcal{X}$ is the identity operator defined on $\mathcal{X}$), then $\varphi \circ f$ is not completely compact.
\end{lemma}
\begin{proof}
we have
\begin{equation*}\begin{split}
&\{y \in M_n(\mathcal{X}): \|y\| \leq 1\} =f_n(\{x \in M_n(\mathcal{Z}): \|x\| \leq 1\}),
\\& \varphi_n(\{y \in M_n(\mathcal{X}): \|y\| \leq 1\}) = (\varphi \circ f)_n(\{x \in M_n(\mathcal{Z}): \|x\| \leq 1\}),
\end{split}\end{equation*}
by the definition of complete compactness $\varphi \circ f$ is not completely compact.
\end{proof}

\section{Some properties of compact
\\ Schur A-multiplier}
In this section, we will prove that if  $(X, \mu)$ and $(Y, \upsilon)$ are standard measure spaces, then there is no non-trivial compact Schur A-multiplier if either $(X, \mu)$ or $(Y, \upsilon)$ is non-atomic. In the following $A$ $\subseteq \mathcal{O}(H)$ is $C^{\ast}$-algebra, and we fix admissible topologies on $X$ and $Y$ respectively.

\bigskip

\begin{lemma} \label{lemma 1}
Let D be any compact subset of $X \times Y$ with $(\mu \times \upsilon) (D) >0$. Then for arbitrary positive number $\epsilon>0$, there are $\mu$-measurable subset $D_X$ of $X$ and $\upsilon$-measurable subset $D_Y$ of $Y$ such that
\begin{equation*}\begin{split}
& (\mu \times \upsilon )((D_X \times D_Y) \setminus D) < \epsilon \cdot (\mu \times \upsilon)(D_X \times D_Y)<{\infty}.
\end{split}\end{equation*}
Furthermore, if $(X, \mu) ( resp. (Y, \upsilon))$ is non-atomic, there are infinitely many mutually disjoint $\mu$ -measurable subsets  $\{D_n\}_{n \in \mathbb{N} }$ of $D_X$(resp. there are infinitely many mutually disjoint $\upsilon$ -measurable subsets  $\{C_n\}_{n \in \mathbb{N} }$ of $D_Y$ ), such that
\begin{equation*}\begin{split}
& (\mu \times \upsilon )((D_n \times D_Y) \setminus D) < \epsilon \cdot (\mu \times \upsilon)(D_n \times D_Y) <{\infty},
\\& (resp. \  (\mu \times \upsilon )((D_X \times C_n) \setminus D) < \epsilon \cdot (\mu \times \upsilon)(D_X \times C_n) <{\infty})
\end{split}\end{equation*}
for all $n \in \mathbb{N}$.
\end{lemma}
\begin{proof}
Let $0< \epsilon <1$ be a given positive number, we choose a number $\delta$ such that $0< \delta < \epsilon \cdot (\mu \times \upsilon )(D)$. By the construction of the product measures (see ~\cite{MR1681462}), there exists a set $\{ V_n \times W_n: n \in  \mathbb{N} \}$ of disjoint rectangles such that $D \subset \cup_{n=1}^{\infty}  V_n \times W_n $ and
\begin{equation*}\begin{split}
\sum _{n=1}^{\infty} \ (\mu \times \upsilon )(V_n \times W_n)< (\mu \times \upsilon )(D)+ \delta.
\end{split}\end{equation*}
It is easy to see that there is at least one $n \in \mathbb{N}$ such that
\begin{equation*}
(\mu \times \upsilon )((V_n \times W_n) \setminus D)< \epsilon \cdot \mu \times \upsilon(V_n \times W_n).
\end{equation*}

\bigskip
Now suppose $(X, \mu)$ is non-atomic. Let $n \in \mathbb{N} $ be such that
\begin{equation}\begin{split}\label{definitionofn}
&\quad \frac{1}{n} \cdot (\mu \times \upsilon )(D_X \times D_Y)
\\& < \epsilon \cdot  (\mu \times \upsilon )(D_X \times D_Y)-(\mu \times \upsilon)((D_X \times D_Y) \setminus D) ,
\end{split}\end{equation}
and let $\{E_k\}_{k=1}^{n}$ be disjoint $\mu$-measurable subsets of $D_X$ such that $\mu(E_k)=\frac{1}{n} \mu(D_X)$ for each $k$ (for the existence of these sets, see ~\cite[I.4]{probability}). Thus we have
\begin{equation*}
\mu(D_X)= \mu(\cup_{k=1}^n E_k).
\end{equation*}
We claim that there are at least two distinct numbers $r, \ m \in \mathbb{N}$ such that
\begin{equation*}
(\mu \times \upsilon )((E_k  \times D_Y)\setminus D) \le \epsilon \cdot (\mu \times \upsilon )(E_k \times D_Y), k=r, m.
\end{equation*}
because if this was not true we would have
\begin{equation*}\begin{split}
\epsilon \cdot  (\mu \times \upsilon )(D_X \times D_Y)-(\mu \times \upsilon )((D_X \times D_Y) \setminus D) & <  (\mu \times \upsilon )(E_n \times D_Y)
\\&=\frac{1}{n} \cdot (\mu \times \upsilon )(D_X \times D_Y),
\end{split}\end{equation*}
this contradicts  (\ref{definitionofn}). This contradiction proved the existence of the two distinct numbers $r$ and $m \in \mathbb{N}$. Now let $D_1=E_r$. We replace $D_X$ by $E_m$ and repeated the same argument, our proof is completed.
\end{proof}

\bigskip

We list the following two lemmas for reference, their proofs are routine. 

\bigskip
\begin{lemma} \label{lemma2}
Let $\varphi: X \times Y \to CB(A, \mathcal{O}(H))$ be a Schur-A multiplier such that $\varphi(x,y)(a)=0$ for $(\mu \times \upsilon)$-almost $(x,y) \in X \times Y$ for all $a \in A$, then $S_{\varphi}=0$.
\end{lemma}

\bigskip

\begin{lemma}\label{decompsition of measure}
If $(Z, \lambda)$ is a standard measure space, $C=\{e \in Z: \lambda(\{e\})>0 \}$, then $C$ is countable and $(Z \setminus C, \lambda)$ is non-atomic.
\end{lemma}

\bigskip
\begin{proposition}\label{X is non-atomic}
Let  $(X, \mu)$ or $(Y, \nu)$ be non-atomic standard measure space, and $\varphi: X \times Y \to CB(A, \mathcal{O}(H))$ be a compact Schur A-multiplier, then $\varphi=0$ for $( (\mu \times \nu) )$-alomost all $(x,y) \in X \times Y$.
\end{proposition}
\begin{proof}
We prove that if $(X,\mu)$ is non-atomic, then $S_{\varphi}=0$ if $\varphi$ is compact Schur $A$-multiplier. The other part is proved by the same argument.

By Lemma \ref{lemma2}, there is $a \in A$ with $\|a\|=1$ such that for some positive number $c >0$,  the measure of the set
\begin{equation}\label{(3)}
D=\{(x,y) \in X \times Y: \|\varphi(x,y)(a)\| >c \}
\end{equation}
is positive.
\bigskip
By the Vector-Valued Lusin's Theorem \cite[Corollary B.28]{MR2288954}, there exists a compact subset $E \subset D$ such that $(\mu \times \upsilon) (E)>0$ and the map from $E$ into $\mathcal{O}(H)$ defined by
\begin{equation*}
  (x,y) \mapsto \varphi(x,y)(a),
\end{equation*}
is continuous. Therefore  $\{\varphi(x,y)(a): (x,y) \in E\}$ is compact subset in $\mathcal{O}(H)$. Let $\epsilon$ be a fixed positive number. By \cite[ Lemma B.23]{MR2288954}, there is a function $f$ of the form $f=\sum_{i=1}^n  \chi_{A_i} \otimes a_i$, where $a_i \in \mathcal{O}(H)$ and $A_i \subset E$ is measurable, such that $\|f\|_{\infty} \le \| \varphi \|$ and
\begin{equation*}
\|\varphi(x,y)(a)-f(x,y)\| < \epsilon \ \ \ \ ( (x,y) \in E).
\end{equation*}
By (\ref{(3)}) it is easy to see that there is at least one $A_k$ such that $(\mu \times \upsilon)(A_k) >0$ and we can assume that $A_k$ is compact.

\bigskip
By Lemma \ref{lemma 1}, there are $\mu$-measurable subset $D_X$ and $\upsilon$-measurable subset $D_Y$, such that
\begin{equation}\begin{split}\label{equation (5)}
& (\mu \times \upsilon)((D_X \times D_Y) \setminus A_k) < \epsilon \cdot (\mu \times \upsilon)(D_X \times D_Y) < {\infty}.
\end{split}\end{equation}
Now we define the function $\vartheta: X \times Y \to \mathcal{O}(H)$ in $\mathfrak{L}_2(X \times Y, \mathcal{O}(H))$ by 
\begin{equation*}\begin{split}
\vartheta (x,y)=\chi_{D_X \times D_Y}(x,y) a_k \ \ \ \ ( (x,y) \in X \times Y).
\end{split}\end{equation*}
Let $h \in H$ be such that $\|h\|=1$ and
\begin{equation}\label{equation 9}
\|\vartheta(x,y)(h)\|=\|a_k(h)\| \ge c- \epsilon \ \ \ \ ( (x,y) \in D_X \times D_Y).
\end{equation}

Since $(X, \mu)$ is non-atomic, by Lemma \ref{lemma 1}  there are  infinitely many mutually disjoint $\mu$-measurable subsets $\{D_n\}_{n \in \mathbb{N} }$ of $D_X$ such that
\begin{equation*}
(\mu \times \upsilon)((D_n \times D_Y) \setminus A_k) < \epsilon \cdot \mu \times \upsilon((D_n \times D_Y)) <{\infty}.
\end{equation*}
Define
\begin{equation}\label{euqation 8.1}
k_n(x,y):=\frac{a}{\mu (D_n)^{\frac{1}{2}}  \cdot  \upsilon (D_Y)^{\frac{1}{2}}}   \chi_{D_n \times D_Y} (x,y),
\xi_n(x):= \frac{h}{\mu (D_n)^{\frac{1}{2}}} \chi_{D_n} (x)
\end{equation}
Then $\{k_n \}_{n \in \mathbb{N} } \in \mathfrak{L}_2(X \times Y, \mathcal{O}(H))$ and  $\|k_n\|_2=\|a\|=1$ ($n \in \mathbb{N}$); $\{ \xi_n \}_{n \in \mathbb{N}} \subset \mathfrak{L}_2(X, H)$ and $\|\xi_n\|=1$ ($n \in \mathbb{N}$). So $\{ \|T_{k_n} \| \}_{n \in \mathbb{N}}$ is bounded. We can complete our proof by showing that $\{T_{\varphi \cdot k_n}\}_{n \in \mathbb{N} }$ has no Cauchy subsequence if the given $\epsilon$ is small enough. By (\ref{equation (5)}),  (\ref{equation 9}), Cauchy-Schwarz inequality and that $\|(T_{\varphi \cdot k_n}-T_{\varphi \cdot k_m})(\xi_n)\| \leq \|T_{\varphi \cdot k_n}-T_{\varphi \cdot k_m}\|$, it is rountine to verify that if $\epsilon < (1/100) \cdot c$ we have
\begin{equation*}\begin{split}
& \text{\Large{$\|$}}(T_{\varphi \cdot k_n}-T_{\varphi \cdot k_m})\text{\Large{$\|$}} \ge \text{\Huge{(}}\int_{D_Y}\text{\Large{$\|$}}\int_{D_n} \frac{1}{\mu (D_n)  \cdot  \upsilon (D_Y)^{\frac{1}{2}}} \vartheta(x,y)(h) d \mu x\text{\Large{$\|$}}^2 d \upsilon\text{\Huge{)}}^{\frac{1}{2}}
\\&  \quad - \text{\Huge{(}}\int_{D_Y}\text{\Large{$\|$}}\int_{D_n} \frac{1}{\mu (D_n)  \cdot  \upsilon (D_Y)^{\frac{1}{2}}}  (\varphi(x,y)(a)- \vartheta (x,y))(h) d \mu x\text{\Large{$\|$}}^2d \upsilon y \text{\Huge{)}}^{\frac{1}{2}}
\\& \ \ \ \ \ \ \ \ \ \ \ \ >  \frac{1}{2} \cdot c,
\end{split}\end{equation*}
our proof is complete.
\end{proof}

\section{Haagerup tensor products and completely compact maps}

Let $K$ be a fixed Hilbert space. We will study the connection between  Haagerup tensor product and completely compact maps.

If $\{A_i\}_{i \in I}$ is a collection of $C^{\ast}$-algebras, we  denote their $C_0$ (or it is called $C^{\ast}$)- direct sum by $\sum_{i \in I}^{\oplus0} A_i$ (see ~\cite{MR936628}).

For a $C^{\ast}$-algebra $A$, we follow Fell and Doran ~\cite{MR936628}  to call $A$ elementary $C^{\ast}$-algebra if $A$ is $\ast$-isomorphic to $\mathcal{O}_c(H)$ for some Hilbert space $H$; on the other hand, we call $A$ compact type $C^{\ast}$-algebra if $A$ is $\ast$-isomorphic to a subalgebra of $\mathcal{O}_c(H)$ for some Hilbert space $H$. If $A$ is compact type $C^{\ast}$-algebra, we shall identify $A=\sum_{i \in I} ^{\oplus 0} \mathcal{O}_c(H_i)$ for some Hilbert spaces  $H_i$ ($i \in I$) (~\cite[Theorem VI.23.3]{MR936628}).

\begin{lemma}\label{inversecompact}
If $A$ is a compact-type $C^{\ast}$-algebra then the following are equivalent:

(i) A is elementary.

(ii) For any $\digamma \in \mathfrak{CCO}(A)$, there are families $\{a_i\}_{i \in J}$ and $\{b_j\}_{j \in J}$ of elements of $A$ such that $\sum_{j \in J} a_i ^{\ast} a_i$ and $\sum_{j \in J} b_j b_j^{\ast}$ are convergent and
\begin{equation*}
\digamma(r)= \sum_{j \in J} b_j \ r \ a_j \ \ \ \ (r \in A).
\end{equation*}
If these conditions hold, we have $\mathfrak{CCO}(A)=A \otimes_h A$.
\end{lemma}
\begin{proof}
The implication from (i) to (ii) is \cite[Corollary 3.6]{compactness}.

(ii) implies (i) :Let $A= \sum_{i \in I}^{\oplus0} \mathcal{O}_c(H_i)$ for some collection $\{H_i\}_{i \in I}$ of Hilbert spaces, we claim that $I$ is single point.

Let $H=\sum_{i \in I}^{\oplus} H_i$, if we represent any element $r$ of $\mathcal{O}(H)$ by a matrix $(r_{i,j})_{i,j \in I}$, where $r_{i,j} \in \mathcal{O}(H_j, H_i)$, then $s \in A$ is a diagonal matrix such that $s_{i,i} \in \mathcal{O}_c (H_i)$. Furthermore, for any $i,j \in I$ we define a map $E_{i,j}: \mathcal{O}_c(H_j, H_i) \to \mathcal{O}_c(H)$ by the following way: for any $a \in \mathcal{O}_c(H_j, H_i)$, $E_{i,j}(a)$ is the matrix in $\mathcal{O}_c(H)$ whose all entries are 0 except for the $i,j$-th entry, which  is $a$. Now we take two distinct points $i_1, i_2$ in $I$, let $a \in \mathcal{O}_c(H_{i_1})$, $b \in \mathcal{O}_c(H_{i_2}, H_{i_1})$, and $c \in \mathcal{O}_c(H_{i_1}, H_{i_2})$ be all non-zero. We define $\Lambda: \mathcal{O}_c(H) \to \mathcal{O}_c(H)$ by 
\begin{equation*}
\Lambda(r)=(E_{i_1, i_1}(a) + E_{i_1, i_2}(b)) \  r \ (E_{i_1, i_1}(a)+ E_{i_2,i_1}(c)) \ \ \ \ (r \in \mathcal{O}_c(H)),
\end{equation*}
then $\Lambda$ is completely compact by \cite[Corollary 3.6]{compactness}, and it is easy to verify that $\Lambda(A) \subset A$. Let $\digamma: A \to A$ be defined by $\digamma(r)=\Lambda(r)$ $(r \in A)$, by Remark \ref{hfsadfjklhvuivrioreiu} $\digamma$ is completely compact. But $\digamma \neq \phi_v$ for any $v \in A \otimes_h A$  because $\phi_v (\mathcal{O}_c(H_i)) \subset \mathcal{O}_c(H_i)$ for any $v \in A \otimes A$ and $i \in I$. This contradiction proved that $I$ is single point, $A$ is elementary. 
\end{proof}

\begin{theorem}\label{new}
If $A$ is a $C^{\ast}$-algebra, then the following are equivalent:

(i) A is elementary,

(ii) For any $\digamma \in \mathfrak{CCO}(A)$, there are families $\{a_i\}_{J}$ and $\{b_j\}_{j \in J}$ of elements of $A$ such that $\sum_{j \in J} a_i ^{\ast} a_i$ and $\sum_{j \in J} b_j b_j^{\ast}$ are convergent and
\begin{equation*}
\digamma(r)= \sum_{j \in J} b_j \ r \ a_j.
\end{equation*}
If these conditions hold, we have $\mathfrak{CCO}(A)=A \otimes_h A$.
\end{theorem}
\begin{proof}
The implication from (i) to (ii) is ~\cite[Corollary 3.6]{compactness}.

(ii) implies (i): In particular, for any $u \in A$, the map $x \mapsto uxu$ is compact, by ~\cite{MR0296716} there is a faithful $\ast$-representation $\pi$ of $A$ on Hilbert space $X$ such that $\pi(A) \subset \mathcal{O}_c(X)$,  thus if we identify $A$ with its image in $\mathcal{O}(X)$, we can consider that $A$ is a norm-closed $\ast$-subalgebra of $\mathcal{O}_c(X)$. By Lemma \ref{inversecompact} (i) holds.

\end{proof}

\section{Compactness and Haagerup tensor products}

By the results of  Section 3, the only interesting compact Schur $A$-multipliers are defined on $X, \ Y=\mathbb{N}$, equipped with the counting measure. We assume that $A \subset \mathcal{O}(H)$ for some Hilbert space $H$, and we can drop the assumption that $A$ is separable.

We identify each $T \in \mathcal{O}(H^{\infty})$ with a matrix $(T_{m,n})_{m,n \in \mathbb{N}}$, where $T_{m,n} \in \mathcal{O}(H)$. Furthermore, we define the conditional expectation $\mathcal{E}: \mathcal{O}(\ell^2) \to \ell^{\infty}$ by
\begin{equation*}
\mathcal{E}(S)(n)=S_{n,n},\rm{ \ for \ all \ matrix \ S \in \mathcal{O}(\ell^2} ) .
\end{equation*}
For each $n \in \mathbb{N}$ we define $\mathcal{E}_n: \mathcal{O}(\ell^2) \to \ell^{\infty}$ by
\begin{equation*}\begin{split}
&\mathcal{E}_n(S)(k)=S_{k,k}\ \ \ \ (k \leq n);
\\& \mathcal{E}_n(S)(k)=0 \ \ \ \ \ (k > n) 
\end{split}\end{equation*}
($\rm{ for \ all \ matrix \ S \in \mathcal{O}(\ell^2} )$). Since $\ell^2$ is commutative $C^{\ast}$-algebra, $\mathcal{E}$ is completely bounded. Therefore the action of Schur $A$- multiplier $\varphi$ on $\mathcal{O}_c(\ell^2(\mathbb{N})) \otimes A$ can be regarded with
\begin{equation*}\begin{split}
S_{\varphi} & :  \mathcal{O}_c(\ell^2(\mathbb{N})) \otimes A \rightarrow \mathcal{O}_c(\ell^2(\mathbb{N})) \otimes \mathcal{O}(H)
\\& : (T_{m,n})_{m,n \in \mathbb{N}} \mapsto (\varphi (n,m) (T_{m,n}))_{m,n \in \mathbb{N}}.
\end{split}\end{equation*}

\bigskip

\begin{lemma}\label{matrix product}
Let $S_{\varphi}$ be a $Schur \ A $-multiplier. If there exist an index set $J$, and families of $\{R_i\}_{i \in J}$ and $\{S_i\}_{i \in J}$ $\subset \mathcal{O}(H^{\infty})$ such that $\sum_{i \in J} R_i R_i^{\ast}$ and $\sum_{i \in J} S_i^{\ast}S_i$ are convergent, and
\begin{equation}\label{1841702}
S_{\varphi}(T)=  \sum_{i \in J} R_i T S_i, \ T \in \mathcal{O}_c(\ell^2({\mathbb{N}})) \otimes A,
\end{equation}
 then $\sum_{i \in J} \mathcal{E}(R_i) \mathcal{E}(R_i)^{\ast}$ and $\sum_{i \in J}\mathcal{E}(S_i)^{\ast}\mathcal{E}(S_i)$ are convergent and
\begin{equation}\label{1841701}
S_{\varphi} (T) = \sum_{i \in J} \mathcal{E}(R_i) T \mathcal{E}(S_i),  \ T \in \mathcal{O}_c(\ell^2({\mathbb{N}})) \otimes A.
\end{equation}
\end{lemma}
\begin{proof}
The convergence of  $\sum_{i \in J} \mathcal{E}(R_i)\mathcal{E}(R_i)^{\ast}$ and $\sum_{i \in J} \mathcal{E}(S_i)^{\ast}\mathcal{E}(S_i)$  is an easy convergence of  $\sum_{i \in J} R_i R_i^{\ast}$ and $\sum_{i \in J} S_i^{\ast}S_i$ .
Let $R_i=(a^{(i)}_{m,n})_{m,n \in \mathbb{N}}$, $S_i=(b^{(i)}_{m,n})_{m,n \in \mathbb{N}}$, where $a^{(i)}_{m,n}, b^{(i)}_{m,n} \in \mathcal{O}(H)$. Since $S_{\varphi}$ is linear and continuous, the linear span of $\{E_{p,q}(a): a \in A; p,q \in \mathbb{N}\}$ is norm-dense in $\mathcal{O}_c(\ell^2({\mathbb{N}})) \otimes A$ (here we recall that $E_{p,q}(a)$ is the matrix whose entries are all 0 but $p,q$-th entry is $a$), thus in oder to verify (\ref{1841701}), it is sufficient to verify that it holds for $E_{p,q}(a)$ for any $a \in A$ and $p,q \in \mathbb{N}$. By  (\ref{1841702}) we have ($\cdot$ is the multiplication of matice)
\begin{equation}\begin{split}\label{19011701}
&e_{p,q} \otimes 1_{\mathcal{O}(H)} \cdot (S_{\varphi}(E_{p,q}(a)) \cdot e_{p,q} \otimes 1_{\mathcal{O}(H)}= E_{p,q}(\sum_{i \in J} a^{(i)}_{p,p} a b^{(i)}_{p,q}).
\end{split}\end{equation}
On the other hand we have 
\begin{equation}\label{19011702}
e_{p'',q''} \otimes 1_{\mathcal{O}(H)} \cdot S_{\varphi}(E_{p,q}(a)) \cdot e_{p',q'} \otimes 1_{\mathcal{O}(H)}=0
\end{equation}
for all $p'' \rm{or} \ q''\neq p$, or $p' \rm{or} \ q' \neq q$. Now (\ref{19011701}) and (\ref{19011702}) imply that
\begin{equation*}\begin{split}
S_{\varphi}(E_{p,q}(a))=\sum_{i \in J} \mathcal{E}(R_i) \ E_{p,q}(a) \  \mathcal{E}(S_i)
\end{split}\end{equation*}
Our proof is complete.
\end{proof}

\bigskip

In the sequel part, we use symbol $\mathfrak{CS} (A, \mathcal{O}(H))$ $(resp. \mathfrak{CCS}(A,\mathcal{O}(H)))$ to denote the set of compact $(resp. completly \ compact)$ Schur $A$-multiplier. Furthermore, we define $\mathfrak{CS}(A)=\mathfrak{CS}(A,A) (resp. \mathfrak{CCS}(A)=\mathfrak{CCS}(A,A)) $.

\begin{remark}\label{usefulkey}
Recall the discussion in section 1.2, we have
\begin{equation*}
\mathfrak{CO} (A, \mathcal{O}(H)) \subset \mathfrak{CS} (A, \mathcal{O}(H)),
\end{equation*}
and
\begin{equation*}
\mathfrak{CCO}(A,\mathcal{O}(H)) \subset \mathfrak{CCS}(A,\mathcal{O}(H)),
\end{equation*}
\end{remark}

\begin{proposition}\label{maintheorem}
$\mathfrak{CCS}(\mathcal{O}_c(H))=c_0(\mathbb{N}, \mathcal{O}_c(H)) \otimes_h c_0(\mathbb{N}, \mathcal{O}_c(H))$
\end{proposition}
\begin{proof}
Suppose $\varphi$ is Schur $\mathcal{O}_c(H)$-multiplier. Each $T \in \mathcal{O}_c(\ell^2(\mathbb{N})) \otimes \mathcal{O}_c(H)$ may be identified with a matrix $(T_{m,n})_{m,n \in \mathbb{N}}$, where $T_{m,n} \in \mathcal{O}_c(H)$ for all $m,n \in \mathbb{N}$. By Lemma \ref{inversecompact} there exist an index set $J$ and $\{S_i \}_{i \in J} \subset \mathcal{O}_c(H^{\infty})$, $\{R_i \}_{i \in J} \subset \mathcal{O}_c(H^{\infty})$ such that $\sum_{ i \in J} R_i R^{\ast}_{i \in J}$ and $\sum_{i \in J}S_i^{\ast}S_i$ converge uniformly and
\begin{equation*}
S_{\varphi}(T)=\sum_{i \in J}R_i T S_i, T \in \mathcal{O}_c(H^{\infty}).
\end{equation*}
By Lemma \ref{matrix product} we have
\begin{equation}\label{diagoanl form of multiplier}
S_{\varphi}(T)= \sum_{i \in J} \mathcal{E}(R_i) T \mathcal{E}(S_i),  \ T \in \mathcal{O}_c(\ell^2(\mathbb{N})) \otimes A.
\end{equation}
Now it is easy to verify that $\{\mathcal{E}(R_i)\}_{i \in J}$, $\{\mathcal{E}(S_i)\}_{i \in J}$ are collections of compact operators, $\sum_{i \in J} \mathcal{E}(R_i) \mathcal{E}(R_i)^{\ast}$ and $\sum_{i \in J} \mathcal{E}(S_i)^{\ast} \mathcal{E}(S_i)$ converge in norm. By \cite{MR1138841}, let $v=\sum_{i \in J}\mathcal{E}(R_i) \otimes \mathcal{E}(S_i) \in c_0(\mathbb{N}, \mathcal{O}_c(H)) \otimes_h c_0(\mathbb{N}, \mathcal{O}_c(H))$, we have $\|S_{\varphi}\|_{cb}$=$\|v\|_h$.

Conversely,  for any $v \in c_0(\mathbb{N},\mathcal{O}_c(H)) \otimes_h c_0(\mathbb{N},\mathcal{O}_c(H))$, there are $\{R'_k\}_{k \in \mathbb{N}},\{S'_k\}_{k \in \mathbb{N}}$ in $\mathcal{O}_c(H)$ such that $v=\sum_{k=1}^{\infty}R'_k \otimes S'_k$ and $\text{\Large{$\|$}}\sum_{k=1}^{\infty} R'_k {R'_k}^{\ast}\text{\Large{$\|$}} < + \infty$, $\text{\Large{$\|$}}\sum_{k=1}^{\infty}{S'_k}^{\ast}S'_k\text{\Large{$\|$}}< + \infty$, then
\begin{equation*}
(\phi_v)|(\mathcal{O}_c(H^{\infty})) , T \mapsto \sum_{k=1}^{\infty} R'_k T S'_k
\end{equation*}
is a completely compact map on $\mathcal{O}_c(H^{\infty})$, and it is easy to see that there is Schur $\mathcal{O}_c(H)$-multiplier $\varphi$ such that $S_{\varphi}=\phi_v$.  Therefore the map
\begin{equation*}
C_0(\mathbb{N},\mathcal{O}_c(H)) \otimes_h C_0(\mathbb{N},\mathcal{O}_c(H)) \to \mathfrak{CCS}(\mathcal{O}_c(H(H))), v \mapsto (\phi_v)|\mathcal{O}_c(H^{\infty})
\end{equation*}
is linear isometry with range $\mathfrak{CCS}(\mathcal{O}_c(H))$. Our proof is complete.
\end{proof}

\begin{theorem}\label{final1}
If $A$ is $C^{\ast}$-algebra, then the following two conditions are equivalent:

(I)$A$ is elementary;

(II) $\mathfrak{CCS}(A)=c_0(\mathbb{N}, A) \otimes_h c_0(\mathbb{N}, A)$.

If these conditions hold, Schur $A$-multiplier $\varphi$ is completely compact if and only if there are index set $J$, $\{a^i_k\}_{i \in J, k \in \mathbb{N}}$ and $\{b^i_k\}_{i \in J, k \in \mathbb{N}}$ $\subset A$ such that : 

(i) \, $\sum _i a^i_k (a^i_k)^{\ast}$ and $\sum_i (b^i_k)^{\ast} b^i_k $ are convergent in the norm of $A$ for each $k \in \mathbb{N}$, and
\begin{equation*}
\sum_{i \in J} a_k^i(a_k^i)^{\ast}, \sum_{i \in J}(b_k^i)^{\ast}b_k^i \to   0  \ \ {\rm{if}} \ \ k \to \infty 
\end{equation*}

(ii) for any $m,n \in \mathbb{N}$, 
\begin{equation}\label{equation of multiplier}
\varphi(m,n)( x)= \sum_i a^i_n \, x \, b^i_m \ \ \ \ ( x \in A).
\end{equation}
\end{theorem}
\begin{proof}
The implication from $(I)$ to $(II)$ was proved in the previous proposition.

$(II)$ implies $(I)$: This is the combination of Remark \ref{usefulkey} and Proposition \ref{new}.

Therefore if $(I)$ or $(II)$ holds, we may identify $A=\mathcal{O}_c(H)$ for some Hilbert space $H$. Let $\varphi: \mathbb{N} \times \mathbb{N} \to CB(A)$ be a given Schur $A$-multiplier. We prove that $\varphi$ is completely compact if and only if $(i)$ and $(ii)$ hold.

If $(i)$ and $(ii)$ hold, then the second part of $(i)$ implies that $\|a_k^i \|^2=\| a_k^i (a_k^i)^{\ast}\| \to 0$ (resp. $\|a_k^i \|^2=\|(b^i_k)^{\ast} b^i_k\| \to 0$) for each fixed $i$ if $k \to \infty$.Thus we may define $R_i$ (resp. $S_i$) $\in c_0(\mathbb{N},\mathcal{O}_c(H))$ by $(R_i)_{k,k}=a^i_k$ (resp. $(S_i)_{k,k}=b^i_k$). Now $(i)$ implies that $\sum_{i \in J} R_i R_i^{\ast}$ and $\sum_{i \in J} S_i^{\ast}S_i$ are convergent in $c_0(\mathbb{N},\mathcal{O}_c(H))$, then we conclude that $\sum_i R_i \otimes_h S_i$ is in $c_0(\mathbb{N},\mathcal{O}_c(H)) \otimes_h c_0(\mathbb{N},\mathcal{O}_c(H))$, and $(ii)$ implies that
\begin{equation*}
S_{\varphi}(T)= \sum_i R_i \, T \, S_i \ \ \ \ (T \in \mathcal{O}_c(\ell^2) \otimes \mathcal{O}_c(H)),
\end{equation*}
so by Theorem \ref{maintheorem} $S_{\varphi}$ is completely compact map.

Now suppose that $\varphi$ is completely compact Schur $\mathcal{O}_c(H)$-multiplier, then there is $v=\sum_{k \in \mathbb{N}} R_k \otimes_h S_k$ in $c_0(\mathbb{N}, \mathcal{O}_c(H)) \otimes c_0(\mathbb{N}, \mathcal{O}_c(H))$ such that $S_{\varphi}=(\phi_v)|(\mathcal{O}_c(H^{\infty}))$. Furthermore, $\sum_{i \in \mathbb{N}}R_i R_i^{\ast}$ and $\sum_{i \in \mathbb{N}}S_i^{\ast}S_i$ are convergent in the norm of $\mathcal{O}(H^{\infty})$, we conclude that $\sum_{i \in \mathbb{N}}R_i R_i^{\ast}$, $\sum_{i \in \mathbb{N}}S_i^{\ast}S_i \in C_0(\mathbb{N}, \mathcal{O}_c(H))$. We define $a^i_k=(R_i)_{k,k}$ and $b^i_k=(S_i)_{k,k}$ for each $i, \, k \in \mathbb{N}$, then it is easy to verify that $\{a_k^i\}_{i,k \in \mathbb{N}}$ and $\{b_k^i\}_{i,k \in \mathbb{N}}$ satisfy (1) and (2).
\end{proof}

\section{Applications to compact-type $C^{\ast}$-algebra}

In this section, we will use the results of last section to study the compactness of Schur $A$-multiplier where $A$ is compact-type $C^{\ast}$-algebra. We identify $A=\sum_{i \in I}^{0 \oplus} \mathcal{O}_c(H_i)$ for  some collection $\{H_i\}_{i \in I}$ of Hilbert spaces, take $H=\sum_{i \in I}^{\oplus} H_i$ and let $f: \mathcal{O}_c(H) \to A$ be the canonical projection, thus $f$ is completely positive. We shall say that the pair  $(H,f)$ is $associated$ to $A$.

\begin{lemma}\label{extension}
Let $A$ be compact-type $C^{\ast}$-algebra, $(H,f)$ its associated pair. If $\varphi: \mathbb{N} \times \mathbb{N} \to CB(A)$ is Schur $A$-multiplier, then there is a Schur $\mathcal{O}_c(H)$-multiplier $\psi:  \mathbb{N} \times \mathbb{N} \to CB(\mathcal{O}_c(H))$ such that $S_{\psi}|(\mathcal{O}_c(\ell^2(\mathbb{N})) \otimes A)=S_{\varphi}$. Moreover, $S_{\varphi}$ is (completely) compact if and only if $S_{\psi}$ may be chosen to be (completely) compact map.
\end{lemma}
\begin{proof}
Since $f$ is completely bounded,  by \cite[Theorem 2.6]{MTT16} $\rho: \mathbb{N} \times \mathbb{N} \to CB(\mathcal{O}_c(H), A)$ defined by
\begin{equation*}
\rho(n,m)=f
\end{equation*} 
is Schur $zmathcal{O}_c(H)$-multiplier. For each $m,n$ let us define $\psi(n,m): \mathcal{O}_c(H) \to \mathcal{O}_c(H)$ by
\begin{equation*}
\psi(n,m)(a)= (\varphi(n,m) \circ \rho(n,m))  \ (a) \ \ \ \ (a \in \mathcal{O}_c(H)).
\end{equation*}
Then $S_{\psi}(T)=S_{\varphi} \circ S_{\rho}(T)$ for all $T \in \mathcal{O}_c(\ell^2)\otimes \mathcal{O}_c(H)$, $S_{\psi}:  \mathcal{O}_c(\ell^2)\otimes \mathcal{O}_c(H) \to  \mathcal{O}_c(\ell^2)\otimes \mathcal{O}_c(H)$ is Schur $\mathcal{O}_c(H)$-multiplier whose range is contained in $\mathcal{O}_c(\ell^2) \otimes A$ and  $S_{\psi}(T)=S_{\varphi}(T)$ for all $T \in \mathcal{O}_c(\ell^2)\otimes A$. Now since $S_{\rho}$ is completely bounded, we conclude that if $S_{\varphi}$ is (completely) compact then $S_{\psi}$ is (completely) compact. Conversely, if $S_{\psi}$ is (completely) compact,   since $S_{\psi}(T)=S_{\varphi}(T)$ for all $T \in \mathcal{O}_c(\ell^2) \otimes A$ and that there is completely positive map $id \otimes f :  \mathcal{O}_c(\ell^2) \otimes \mathcal{O}_c(H) \to \mathcal{O}_c(\ell^2) \otimes A$ which is identity map on $\mathcal{O}_c(\ell^2) \otimes A$, we conclude that $S_{\varphi}$ is (completely) compact.
\end{proof}

By Lemma \ref{extension} and Theorem \ref{final1} together we have:
\begin{theorem}\label{compacttypecase}
If $A$ is compact type $C^{\ast}$-algebra, then there is Hilbert space $H$ such that the following are equivalent:
\\ (i) $\varphi$ is in $\mathfrak{C.C.S}$(A);
\\ (ii) there are index set $J$ and $\{a^i_k\}_{i \in J, k \in \mathbb{N}}$ and $\{b^i_k\}_{i \in J, k \in \mathbb{N}}$ $\subset \mathcal{O}_c(H)$ such that:

\ \ \ \ \ (1) \, $\sum _i a^i_k (a^i_k)^{\ast}$ and $\sum_i (b^i_k)^{\ast} b^i_k $ are convergent in the norm of $\mathcal{O}(H)$ for each $k \in \mathbb{N}$ and $\sum_{i \in J} a_k^i(a_k^i)^{\ast}, \sum_{i \in J}(b_k^i)^{\ast}b_k^i \to 0 $ as $k \to \infty$; 

\ \ \ \ \ (2) for any $m,n \in \mathbb{N}$,
\begin{equation*}
\varphi(m,n) ( x) =\sum_i a^i_n \, x \, b^i_m, x \in A.
\end{equation*}
\end{theorem}

\begin{remark}
If we compare the previous theorem with Theorem \ref{final1}, we noticed that in condition (ii)  $\{a^i_k\}_{i \in J, k \in \mathbb{N}}$ and $\{b^i_k\}_{i \in J, k \in \mathbb{N}}$ can be chosen from $A$ if and only if $A$ is $\ast$-isomorphic to $\mathcal{O}_c(K)$ for some Hilbert space $K$.
\end{remark}

\section{Compactness and complete compactness}

In this section, we prove some results by aid of which we could identify compactness and complete compactness in some cases. We fix $\Omega$ to be a compact Hausdorff space, $C(\Omega)$ to be the space of all continuous complex-valued functions on $\Omega$.

The proofs of the following two lemmas are standard and we ommit them:
\begin{lemma}\label{power}
Let $\mathcal{X}$ and $\mathcal{Y}$ be operator spaces, $\Psi \in \mathcal{O}_c(\mathcal{X}, \mathcal{Y})$ . If $(\Phi_{\alpha})_{\alpha \in \mathbb{A}} \subset B(\mathcal{Y})$ is a net with $sup_{\alpha \in \mathbb{A}} \text{\Large{$\|$}} \Phi_{\alpha}\text{\Large{$\|$}} < \infty$ and $\Phi \in B(\mathcal{Y})$ such that $\Phi_{\alpha}(a) \to \Phi(x)$ for all $x \in \mathcal{X}$. Then $\text{\Large{$\|$}}\Phi_{\alpha} \circ \Psi-\Phi \circ \Psi\text{\Large{$\|$}} \to _{\alpha \in \mathbb{A}} 0$.
\end{lemma}

\begin{lemma}\label{1842401}
For any $(f_{i,j})_{i,j=1}^{\infty} \in \mathcal{O}(\ell^2) \otimes_{min} C(\Omega)$, we have
\begin{equation*}
\|(f_{i,j})_{i,j=1}^{\infty}\|=sup\{\|(f_{i,j}(\omega))_{i,j=1}^{\infty}\|: \omega \in \Omega\},
\end{equation*}
where the norm of the right hand-side is taken from $\mathcal{O}(\ell^2)$. In particular, $\|(f_{i,j}(\omega))_{i,j=1}^{\infty}\| \leq\|(f_{i,j})_{i,j=1}^{\infty}\|$ .
\end{lemma}

\begin{lemma}\label{18051601}
If $A$ is a $C^{\ast}$-algebra which is $\ast$-isomorphic to a subalgebra of $M_n \otimes C(\Omega)$ for some $n \in \mathbb{N}$ and compact space $\Omega$, then compact linear map from $A$ into $A$ is completely compact.
\end{lemma}
\begin{proof}
Let $\varphi: A \to A$ be a compact linear map. We identify $A$ as a $C^{\ast}$-subalgebra of $ M_n \otimes C(\Omega)$. Since $M_n \otimes C(\Omega)$ is nuclear, let $\{\psi_m: M_n \otimes C(\Omega) \to M_n \otimes C(\Omega)\}_{m \in I}$ be a net of finite-rank completely positive maps such that $\psi_m(a) \to a$ for all $a \in M_n \otimes C(\Omega)$ provided $m \to \infty$. Then we have $\psi_{m} \circ \varphi(a) \to \varphi(a)$ for all $a \in A$. Since $\{\|\psi_m \circ \varphi\|\}$ is bounded, by Lemma \ref{power} we have $\|\psi_m \circ \varphi -\varphi\| \to 0$. Therefore $\{\psi_m \circ \varphi\}$ is Cauchy net in $\mathcal{O}(A, M_n \otimes C(\Omega))$. But $\mathcal{CB}(A, M_n \otimes C(\Omega))=\mathcal{O}(A,M_n \otimes C(\Omega))$, by Open Mapping Theorem we conclude that $\{\psi \circ \varphi\}$ is Cauchy net in $\mathcal{CB}(A,M_n \otimes C(\Omega))$ as well, it is easy to verify that $\psi_m \circ \varphi \to \varphi$ completely. Since each $\psi_m \circ \varphi$ is of finite rank,  $\varphi$ is completely compact.
\end{proof}

\begin{proposition}
Let $V$ be an operator space, and $n \in \mathbb{N}$ be fixed number. Let  $\varphi_{i,j}: V \to C(\Omega)$ be completely bounded maps ($i,j=1, \ldots ,n$), we define $S_{\varphi}: M_n(V) \to M_n(C(\Omega))$ by
\begin{equation*}
S_{\varphi}((v_{i,j})_{i,j=1}^n)=(\varphi_{i,j}(v_{i,j}))_{i,j=1}^n.
\end{equation*}
Then $\|S_{\varphi}\|_{cb}=\|S_{\varphi}\|$.
\end{proposition}
\begin{proof}
Let $\mathcal{D}_n$ be the subalgebra of all the diagonal matrices in $M_n$. We will use the idea of the proof of \cite[proposition 8.6]{Paulsen} . For any $C \in \mathcal{D}_n$, we have
\begin{equation*}\begin{split}
& S_{\varphi}(C(v_{i,j})_{i,j=1}^n)=C(\varphi_{i,j}(v_{i,j}))_{i,j=1}^n,
\\&  S_{\varphi}((v_{i,j})_{i,j=1}^n C)=(\varphi_{i,j}(v_{i,j}))_{i,j=1}^n C
\end{split}\end{equation*}
where we define the multiplication between $C$ and elements of $M_n(V)$ or $M_n({C(\Omega)})$ by the multiplication of matrices. 
 Since $S_{\varphi}$ is $\mathcal{D}_n$-bimodule map, by the similar argument of \cite[Proposition 8.6]{Paulsen} we can prove that $S_{\varphi}$ is completely bounded.
\end{proof}

Now we may get the following proposition immediately:

\begin{proposition}\label{18050701}
Let $V$ be an operator space, $A$ a $C^{\ast}$-algebra which is $\ast$-isomorphic to a subalgebra of $M_n \otimes C(\Omega)$ for some $n \in \mathbb{N}$. Let $\varphi_{i,j}: V \to A$ be a bounded linear map (so it is completely bounded automatically) for each $i,j \in \mathbb{N}$, if $S_{\varphi}: \mathcal{O}_c(\ell^2) \otimes V \to \mathcal{O}_c(\ell^2)  \otimes A$ is a bounded linear map which satisfies
\begin{equation*}
S_{\varphi}((v_{i,j})_{i,j=1}^n)=(\varphi_{i,j}(v_{i,j}))_{i,j=1}^n,
\end{equation*}
then $S_{\varphi}$ is completely bounded and $\|S_{\varphi}\|_{cb} = \|S_{\varphi}\|$.
\end{proposition}

Combine Lemma \ref{18051601} and Proposition \ref{18050701} we get:
\begin{proposition}\label{18051602}
If $A$ is a $C^{\ast}$-algebra which is $\ast$-isomorphic to a subalgebra of $M_n \otimes C(\Omega)$ for some $n \in \mathbb{N}$ and compact space $\Omega$, then the compact Schur $A$-multipliers are completely compact.
\end{proposition}

\bigskip

By modifying the proof of \cite{completelycompact}, it is not hard to prove:

\begin{lemma}\label{counterexample}
Let $H$ be a Hilbert space with infinite dimension, $I$ and index set which has the same cardinal number of $H$ when we regard $H$ as a set.  Let us select an arbitrary pairwise orthogonal family $\{H_i\}_{i \in I}$  of finite-dimensional subspaces of $H$ such that $\sum_{i \in I} H_i=H$ and that there is a subset $\{i_k\}_{k \in \mathbb{N}} \subset I$ such that $dim(H_k)>k$, then there exists a linear map $\varphi: \mathcal{O}_c(H) \to \mathcal{O}_c(H)$ such that $\varphi(\mathcal{O}(H_k)) \subset \mathcal{O}(H_k)$ and that $\varphi$ is completely bounded, compact, but not completely compact.
\end{lemma}

The lemma is similar to the previous:
\begin{lemma}\label{counterexample2}
Let $H$ be a Hilbert space with infinite dimension, $I$ and index set which has the same cardinal number of $H$ when we regard $H$ as a set.  Let us select an arbitrary pairwise orthogonal family $\{H_i\}_{i \in I}$  of finite-dimensional subspaces of $H$ such that $\sum_{i \in I} H_i=H$ and that there is a subset $\{i_k\}_{k \in \mathbb{N}} \subset I$ such that $dim(H_k)>k$, then there exists a linear map $\theta: \sum_{i \in I}^{\oplus 0} \mathcal{O}(H_i)
 \to \sum_{i \in I}^{\oplus0} \mathcal{O}(H_i)$ such that $\theta(\mathcal{O}(H_i)) \subset \mathcal{O}(H_i)$ and that $\theta$ is completely bounded, compact, but not completely compact.
\end{lemma}

\begin{proposition}\label{1}
Let $A$ be a compact-type $C^{\ast}$-algebra such that all compact completely bounded  linear map $\varphi: A \to A$ is completely compact, then  $A$ is $\ast$-isomorphic to $\sum_{k \in I} ^{\oplus0}  M_{n_k}$, $n_k \leq N$ for some $N \in \mathbb{N}$, where $I$ is an index set.
\end{proposition}
\begin{proof}

We assume  A= $\sum_{k \in I} ^{\oplus0}  \mathcal{O}_c(H_k)$,  $H=\sum^{\oplus}_{k \in I}H_k$. So there are three cases:

$(1)$ $H_k$ is of infinite dimension for some $k$;

$(2)$ Each $H_k$ is of finite-dimension but $\{{\rm{dim}}(H_k)\}_{k \in I}$ is unbounded, so it is convenient to assume that $ A=\sum_{k \in I} ^{\oplus0}  M_{n_k} $, and there is a subset $ \{n_{k_i}\}$ of $\{n_k: k \in I\}$ such that $i \leqq n_{k_i}$ for all $i \in \mathbb{N}$;

$(3)$ A= $\sum_{k \in I}^{\oplus0}  M_{n_k}$, $n_k \leq N$ for some $N \in \mathbb{N}$.

We need to prove that (1) and (2) are not true.

\bigskip
(1) is failed: Suppose $H_k$ is of infinite-dimension. By Lemma \ref{counterexample}, there is a map $\varphi: \mathcal{O}_c(H_k) \to \mathcal{O}_c(H_k)$ which is completely bounded, compact but not completely compact. Let $g: \sum_{i \in I} ^{\oplus0}  \mathcal{O}_c(H_i) \to \mathcal{O}_c(H_k)$ be the canonical extension of $\varphi$, that is, $g|\mathcal{O}_c(H_k)$=$\varphi$ and $g| \sum_{i \neq k}^{\oplus0} \mathcal{O}_c(H_i)=0$, then $g$ is compact, completely bounded, but by Lemma \ref{negativecriteria} it is easy to see that $g$ is not completely compact. But $g$ can be regarded as a compact, completely bounded but not completely compact linear map from $\sum_{i \in \mathbb{N}}^{\oplus0}\mathcal{O}_c(H_i)$ into $\sum_{i \in \mathbb{N}}^{\oplus0}\mathcal{O}_c(H_i)$, so (1) is failed.

\bigskip

(2) is failed: In this case, let $B= \sum_{i \in \mathbb{N}}^{\oplus0} M_i$. By Lemma \ref{counterexample}, there is a completely bounded linear map $\varphi: B \to B$ satisfing $\varphi(M_i) \subset M_i$ which is compact but not completely compact. But $B=\sum_{i \in \mathbb{N}}^{\oplus0} M_i$ is a norm-closed ${\ast}$-subalgebra of $A=\sum_{k \in I} ^{\oplus0} M_{n_{k}}$ (since $i \leq n_{k_i}$, $M_i \subset M_{n_{k_i}}$), and there is conditional expectation $E$ from $A$ to $B$. So by Lemma \ref{negativecriteria} it is easy to see that $\varphi \circ E: \sum_{k \in I} ^{\oplus0}  M_{n_k} \to B$ is compact and completely bounded, but it is not completely compact. Furthermore, $\varphi \circ E$ can be regarded as a linear map from $A$ into $A$ which is compact, completely bounded but not completely compact, so (2) is failed.
\end{proof}

\begin{proposition}\label{identification of c and cc}
Let  A= $\sum_{k \in I} ^{\oplus0}  M_{n_k}$ such that $n_k \leq N$ for some $N \in \mathbb{N}$, then for any $C^{\ast}$-algebra $B$, the linear map $\varphi: B \to A$ is compact if and only if $\varphi$ is completely compact.
\end{proposition}
\begin{proof}
Since $A$ is $\ast$-isomorphic to a $C^{\ast}$-subalgebra of $M_N \otimes C(I \cup \{\infty\})$, and $I\cup \{\infty\}$ is compact space,  then the statement is an easy consequence of Proposition \ref{18051602}.
\end{proof}

Now we could summarize a theorem as following:

\begin{theorem}\label{18051001}
If $A$ is a compact-type $C^{\ast}$-algebra, then the following is equivalent: (i)Any compact completely bounded bounded linear map $\varphi: A \to A$ is completely compact; (ii)  A= $\sum_{k \in I} ^{\oplus0}  M_{n_k}$, $n_k \leq N$ for some $N \in \mathbb{N}$. If these conditions hold, then for any $C^{\ast}$-algebra $B$, if  linear map $\varphi: B \to \sum_{k \in I} ^{\oplus0}  M_{n_k}$ is compact, then it is completely compact.
\end{theorem}

\bigskip

Now let us go back to study Schur $A$-multiplier defined on $\mathbb{N} \times \mathbb{N}$.

\begin{theorem}\label{akey}
 If $A \subset \mathcal{O}_c(H)$ is $C^{\ast}$-algebra, then $\mathfrak{CS}(A)=\mathfrak{CCS}(A)$ if and only if $A=\sum_{k \in I} ^{\oplus0}  M_{n_k}$ such that $n_k \leq N$ for some $N \in \mathbb{N}$ .
\end{theorem}
\begin{proof}
This is the combination of Proposition \ref{18050701} and Theorem \ref{18051001} because for $A \subset \mathcal{O}_c(H)$, $A=\sum_{k \in I} ^{\oplus0}  M_{n_k} $ implies that $A$ is $\ast$-isomorphic to a subalgebra of $M_n \otimes C(\Omega)$ for some $\Omega$ and $n$.
\end{proof}

\begin{corollary}
If $A$ is a finite dimension $C^{\ast}$-algebra, then $\mathfrak{CS}(A)=\mathfrak{CCS}(A)$.
\end{corollary}

Combine Theorem \ref{akey} and Proposition \ref{maintheorem}, we get  \cite[Proposition 5]{MR1766604}:
\begin{corollary}
$\mathfrak{CS}(\mathbb{C})=c_0(\mathbb{N}, \mathbb{C}) \otimes_h  c_0(\mathbb{N}, \mathbb{C})$.
\end{corollary}

\begin{theorem}\label{mainthe}
If $A$ is a $C^{\ast}$-algebra, the following two conditions are equivalent:

(i) A is $\ast$-isomorphic to $\mathcal{O}(H)$ for some finite dimensional Hilbert space;

(ii) $\mathfrak{CS}(A)=c_0(\mathbb{N}, A) \otimes_h  c_0(\mathbb{N}, A)$.
\end{theorem}
\begin{proof}
The implication from (i) to (ii) is the combination of Theorem \ref{final1} and Theorem \ref{akey}.

(ii) implies (i): By Remark \ref{usefulkey}, for any $u \in A$, the map from $A$ into $A$ defined by $x \mapsto uxu$ is compact, then by Ylinen ~\cite{MR0296716} $A$ is compact-type, thus $A$ has the following form:
\begin{equation*}
A=\sum_{i \in I} ^{\oplus0} \mathcal{O}_c(H_i),
\end{equation*}
where each $H_i$ is Hilbert space, and by Theorem \ref{compacttypecase}, for any $v \in C_0(\mathbb{N}, A) \otimes_h  C_0(\mathbb{N}, A)$, $\phi_v$ is completely compact Schur $A$-multiplier, so condition (ii) implies that
\begin{equation*}
\mathfrak{CS}(A) \subset \mathfrak{CCS}(A),
\end{equation*}
and of course this implies that
\begin{equation*}
c_0(\mathbb{N}, A) \otimes_h  c_0(\mathbb{N}, A)=\mathfrak{CS}(A) = \mathfrak{CCS}(A).
\end{equation*}
Therefore, by Theorem \ref{akey} $A=\sum_{k \in I} ^{\oplus0}  M_{n_k}$, $n_k \leq N$ for some $N \in \mathbb{N}$. On the other hand, by Theorem \ref{final1}, $A$ is elementary $C^{\ast}$-algebra, hence there is at most one $n_k$ is non-zero, (i) is proved.
\end{proof}

\section*{Acknowledgement}
My sincere thanks to my advisor Professor Ivan Todorov for his guidance during this work. I would also like to thank Dr.Stanislav Shkarin, he inspired me to investigate the proof of Lemma \ref{lemma 1}.

\bigskip
Mathematical Sciences Research Centre, Queen's University Belfast, Belfast, BT7 1NN, United Kingdom

Email: whe02@qub.ac.uk

\bibliographystyle{plain}

\bibliography{referencelist}

\end{document}